\documentclass[12pt]{article}
\usepackage{amssymb,amsfonts,amsmath,amsthm,amsrefs,dsfont,hyperref}
\usepackage[letterpaper, portrait, margin=22.5mm]{geometry}

\newcommand{\1}{\mathds{1}}
\newcommand{\0}{\mathds{O}}

\newcommand{\R}{\mathbb{R}}

\newcommand{\N}{\mathbb{N}}

\newcommand{\Bo}{\mathcal{B}}

\newcommand{\Io}{\mathcal{I}}
\newcommand{\Jo}{\mathcal{J}}
\newcommand{\8}{\infty}

\newcommand{\supp}{\mathrm{supp}}
\newcommand{\Co}{\mathcal{C}}
\newcommand{\Fo}{\mathcal{F}}

\newcommand{\Int}{\mathrm{int}}

\usepackage{scalerel}
\DeclareMathOperator*{\bigplus}{\scalerel*{+}{\sum}}

\newcounter{dummy} \numberwithin{dummy}{section}
\newtheorem{theorem}[dummy]{Theorem}
\newtheorem{lemma}[dummy]{Lemma}

\newtheorem{proposition}[dummy]{Proposition}
\newtheorem{corollary}[dummy]{Corollary}
\newtheorem{question}[dummy]{Question}
\theoremstyle{remark}
\newtheorem{remark}[dummy]{Remark}

\begin{document}

\title{Characterizations of the projection bands and some order properties of the lattices of continuous functions}
\author{Eugene Bilokopytov\footnote{Email address bilokopy@ualberta.ca, erz888@gmail.com.}}
\maketitle

\begin{abstract}
We show that for an ideal $H$ in an Archimedean vector lattice $F$ the following conditions are equivalent:
\begin{itemize}
\item $H$ is a projection band;
\item Any collection of mutually disjoint vectors in $H$, which is order bounded in $F$, is order bounded in $H$;
\item $H$ is an infinite meet-distributive element of the lattice $\mathcal{I}_{F}$ of all ideals in $F$ in the sense that $\bigcap\limits_{J\in \mathcal{J}}\left(H+ J\right)=H+ \bigcap \mathcal{J}$, for any $\mathcal{J}\subset \mathcal{I}_{F}$.
\end{itemize}
Additionally, we show that if $F$ is uniformly complete and $H$ is a uniformly closed principal ideal, then $H$ is a projection band. In the process we investigate some order properties of lattices of continuous functions on Tychonoff topological spaces. \\


MSC2020 06E15, 46A40, 46E25.
\end{abstract}

\section{Introduction and preliminaries}

In this article we largely continue the work started in \cite{erz}, where we studied some properties of abstract Archimedean vector lattices by ``modelling'' them on the dense sublattices of $\Co\left(K\right)$, for compact Hausdorff $K$. While the aforementioned paper was mostly dedicated to order continuity of lattice homomorphisms and regularity of sublattices, the present article is mainly concerned with characterization of projection bands. The subobjects of this type correspond to decompositions of the vector lattices into direct sums, and so it is important to track their presence. In particular, the principal projection property (PPP) often plays the role of the minimal ``richness'' assumption about a vector lattice.\medskip

The main results of the paper characterize projection bands in a ``horizontal'' and a ``vertical'' ways. Namely, the former (Theorem \ref{infd}) claims that an ideal $H$ in an Archimedean vector lattice $F$ is a projection band if and only if it is an infinite meet-distributive element of the lattice $\mathcal{I}_{F}$ of all ideals of $F$ in the sense that $\bigcap\limits_{J\in \mathcal{J}}\left(H+ J\right)=H+ \bigcap \mathcal{J}$, for any $\mathcal{J}\subset \mathcal{I}_{F}$; the latter (Theorem \ref{main}) asserts that $H$ is a projection band if and only if for the (disjoint) subsets of $H$ order boundedness in $H$ is equivalent to order boundedness in $F$.\medskip

On the way to proving these results, in Section \ref{s} we take a detour to investigate various properties of the sublattices of the spaces of continuous function on Tychonoff spaces, in particular the structure of ideals on such spaces (where we continue the line of investigation in e.g. \cite{et2,kv1}). In particular, we prove several versions of the Urysohn lemma (Proposition \ref{dul}, Corollary \ref{ssul}, Proposition \ref{od} and Remark \ref{urur}), most of which can then be reinterpreted as representations of sublattices of $\Co\left(X\right)$ as sums of two ideals (Proposition \ref{dech}). We also characterize regular sublattices of $\Co_{b}\left(X\right)$ in terms of the supports of the order dense ideals in them (Proposition \ref{rod}).\medskip

Section \ref{p} is dedicated to projection bands in abstract vector lattices, in particular it contains the two theorems quoted above. On top of that we show that being a projection band is a local property (Proposition \ref{pb}), and find the largest ideal of a vector lattice which has PP (the projection property) or PPP (Corollary \ref{pppp}). We also show that a vector lattice has PP or PPP if it has a majorizing order dense sublattice with PP or PPP, respectively (Corollary \ref{pb2}). As a consequence of Theorem \ref{main}, we show (Corollary \ref{kks}) that a vector lattice has PP if and only if it is ``locally'' a dense sublattice of $\Co\left(K\right)$, for a compact extremally disconnected Hausdorff $K$. Finally, in Corollary \ref{sm} we prove that a uniformly closed principal ideal in a uniformly complete Archimedean vector lattice is a projection band.\medskip

One of the motives permeating this work is that some properties of a vector lattice are ``horizontal'' in nature, in particular completely describable on the level of the compact spaces on which the vector lattice is modelled. An alternative way of discussing these ``horizontal'' properties is in terms of the lattice of the ideals of a vector lattice. In this section we take an opportunity to draw attention to the utility of this language, in particular to the fact (Corollary \ref{com}) that in an Archimedean vector lattice bands and projection bands can be described as pseudo-complemented and complemented elements of the lattice of the ideals, respectively. We also examine the equivalent characterizations of the ideals of a vector lattice (Proposition \ref{id}), and discuss some correspondences between the vector lattice concepts and their lattice-theoretic counterparts. Roughly speaking a vector lattice concept ``restricted'' to the positive cone becomes the corresponding lattice-theoretic concept (Proposition \ref{lid}, Corollary \ref{lidd} and Corollary \ref{idis}). The material of this section is not novel, but it does not seem to be present in the literature in a coherent form.

\subsection{Some facts from lattice theory}

In this subsection we gather some results and concepts from the general lattice theory. All this material can be found in the standard references such as \cite{birk} and \cite{gr}. Everywhere in this section $\left(P,\le\right)$ is a partially ordered set.

We will call $P$ a \emph{join / meet semi-lattice}, if $p\vee q:=\sup\left\{p,q\right\}$ / $p\wedge q:=\inf\left\{p,q\right\}$ exists in $P$, for every $p,q\in P$. If $P$ is both join and meet semi-lattice, it is called a \emph{lattice}. A \emph{complete lattice} is a lattice in which every set has infimum and supremum. Note that if $Q\subset P$ is \emph{join-closed}, i.e. if $p\vee q\in Q$, for every $p,q\in Q$, then $p\vee q=p\vee_{Q} q$ (the latter means the supremum with respect to the partially ordered set $Q$). A \emph{sublattice} of a lattice is a subset which is both meet-closed and join-closed. We will call a sublattice $Q\subset P$ \emph{regular} if $r=\bigvee_{Q} R$ $\Rightarrow$ $r=\bigvee_{P} R$, for $R\subset Q$ and $r\in Q$, and the same for infimum.

Recall that an \emph{ideal} in a lattice $P$ is a nonempty join-closed down-closed subset (i.e. $J\subset P$ is an ideal, if $p,q\in J$, $r\in P$ $\Rightarrow$ $p\vee q, p\wedge r \in J$). We denote the set of all ideals in $P$ by $\Jo_{P}$. It is easy to see that a nonempty intersection of any collection of ideals is an ideal. Hence, $\Jo_{P}\cup\left\{\varnothing\right\}$ is a complete lattice.

We will say that a lattice $P$ is \emph{distributive} if $p\wedge\left(q\vee r\right)=p\wedge q\vee p\wedge r$, for any $p,q,r\in P$. Note that the inequality $\ge$ is always satisfied. This condition is self-dual in the sense that it is equivalent to $p\vee q\wedge r=\left(p\vee q\right)\wedge \left(p\vee r\right)$, for any $p,q,r\in P$. We will say that $P$ is \emph{infinite-join distributive}, if $p\wedge\bigvee R=\bigvee\limits_{r\in R} p\wedge r$, for every $p\in P$ and $R\subset P$ for which $\bigvee R$ exists. In this case $\bigvee Q\wedge\bigvee R=\bigvee\limits_{q\in Q,~ r\in R} p\wedge q$, for any $Q,R\subset P$ with supremums.

\begin{proposition}\label{dec}Let $P$ be a distributive lattice. Then:
\item[(i)] $\Jo_{P}$ is an infinite-join distributive lattice. It is complete if and only if $P$ has the least element.
\item[(ii)] $J\vee_{\Jo_{P}}H=J\vee H:=\left\{j\vee h,~ j\in J,~ h\in H\right\}$ and $J\wedge_{\Jo_{P}}H=J\cap H$, for any $J,H\in\Jo_{P}$.
\item[(iii)] If $\Io\subset \mathcal{J}_{P}$, then $\bigvee_{\Jo_{P}} \Io=\bigvee\Io:=\left\{j_{1}\vee...\vee j_{n},~ j_{k}\in J_{k}\in \Io\right\}$ and $\bigwedge_{\Jo_{P}} \Io=\bigcap\Io$.
\end{proposition}

If $P$ has a minimal element $0$, it is then the least element of $P$. We say that $p,q$ in $P$ are \emph{disjoint} (and denote it $p\bot q$), if $p\wedge q=0$. The \emph{pseudo-complement} $p^{*}$ of $p$ is the greatest element of $P$ disjoint with $p$ (if it exists). Consider the most important special case.

\begin{proposition}\label{compl}
If $P$ is a distributive lattice with a minimum $0$ and maximum $1$, and $p,q\in P$ are such that $p\wedge q=0$ and $p\vee q=1$, then $q=p^{*}$ (in this case we will call $q$ the \emph{complement} of $p$). The set $P^{comp}$ of all complemented elements in $P$ is a Boolean algebra, which is a sublattice of $P$, that contains $0,1$. In particular, $\left(p\wedge q\right)^{*}=p^{*}\vee q^{*}$ and $\left(p\vee q\right)^{*}=p^{*}\wedge q^{*}$, for every $p,q\in P^{comp}$.
\end{proposition}

\begin{theorem}[Glivenko Theorem]\label{pseudo}
Let $P$ be a complete infinite-join distributive lattice. Then:
\item[(i)] $p^{*}$ exists for every $p\in P$, and the set $P^{*}=\left\{p^{*},~p\in P\right\}=\left\{p\in P,~ p=p^{**}\right\}$ is a complete Boolean algebra, which is also meet-closed in $P$.
\item[(ii)] $p\vee_{P^{*}} q= \left(p^{*}\wedge q^{*}\right)^{*}=\left(p\vee q\right)^{**}$, for every $p,q\in P^{*}$; we also have $\left(p\wedge q\right)^{*}=\left(p^{**}\wedge q\right)^{*}$ and $\left(p\wedge q\right)^{**}=p^{**}\wedge q^{**}$, for every $p,q\in P$.
\item[(iii)] If $r\in P$, then the pseudo-complementation $p\mapsto p^{*r}$ in $\left[0,r\right]$ is given by $p^{*r}=p^{*}\wedge r$ and $p^{**r}=p^{**}\wedge r$.
\end{theorem}

\subsection{Ideals of a vector lattice}

We start with recalling some basic facts and concepts from the vector lattice theory. Let $F$ be a vector lattice, whose positive cone is denoted by $F_{+}$. For any set $G\subset F$ we will denote $G_{+}=G\cap F_{+}$. Note that the underlying lattice of a vector lattice is infinite-join and infinite-meet distributive (see \cite[Theorem 1.8]{ab}). Also, by Riesz Decomposition theorem (\cite[Corollary 15.6]{zl}), $\left[0,e\right]+\left[0,f\right]=\left[0,e+f\right]$ and $\left[-e,e\right]+\left[-f,f\right]=\left[-e-f,e+f\right]$, for any $e,f\in F_{+}$. Recall that $F$ is \emph{Archimedean} if $\bigwedge\frac{1}{n}f=0$, for every $f\in F_{+}$.

A linear map $T:F\to E$ between vector lattices is a \emph{homomorphism} if it preserves the lattice operations  (it is enough to require that it preserves the absolute value $\left|f\right|:f\vee -f$). By a \emph{sublattice} of $F$ we will always mean a linear sublattice, i.e. a vector subspace $E$ of $F$, which is closed with respect to the lattice operations (it is enough to require that it contains the absolute value of any of its elements). If on top of that $0\le f\le e\in E$ implies $f\in E$, we say that $E$ is an \emph{ideal} of $F$. Recall that $G\subset F$ is called \emph{solid} if $f\in F$ and $g\in G$ with $\left|f\right|\le\left|g\right|$ yield $f\in G$. It follows from  Riesz Decomposition theorem that if $G,H\subset F$ are solid then $G+H$ is solid with $\left(G+H\right)_{+}=G_{+}+H_{+}$.

\begin{proposition}\label{id}
For $E\subset F$ the following conditions are equivalent:
\item[(i)] $E$ is an ideal in $F$;
\item[(ii)] $E$ is a sublattice with $e\in E_{+}$, $f\in F_{+}$ $\Rightarrow$ $f\wedge e\in E$;
\item[(iii)] $E_{+}$ is a lattice ideal in $F_{+}$, with $2E_{+}\subset E_{+}$ and $E=\left\{e\in F,~ \left|e\right|\in E_{+}\right\}$;
\item[(iv)] $E$ is solid and $E_{+}+E_{+}\subset E$;
\item[(v)] $\left\{\left(e,f\right)\in F\times F,~ e-f\in E\right\}$ is a sublattice of $F\times F$;
\item[(vi)] $E$ is the kernel of a homomorphism with domain $E$.
\end{proposition}
\begin{proof}
(ii)$\Leftrightarrow$(i)$\Rightarrow$(iii) are easy to see. (v)$\Leftrightarrow$(vi) follows from the general construction of a quotient of an algebraic structure.

(iii)$\Rightarrow$(iv): Since $E_{+}$ is a lattice ideal in $F_{+}$, it follows that $E=\left\{e\in F,~ \left|e\right|\in E_{+}\right\}$ is solid; furthermore if $e,f\in E_{+}$, then $e+f\le 2\left(e\vee f\right)\in E_{+}$, and so $e+f\in E$.\medskip

(iv)$\Rightarrow$(i): Since a solid set is always closed with respect to $\left|\cdot\right|$, we only need to prove that $E$ is a linear subspace. If $e,f\in E$, since $E$ is solid we have $\left|e\right|,\left|f\right|\in E_{+}$, from where $\left|e+f\right|\le \left|e\right|+\left|f\right|\in E_{+}$, and using solidness again, we conclude that $e+f\in E$. If $\lambda\in\R$, take $n\in\N$ such that $\left|\lambda\right|\le n$; we have that $\left|\lambda e\right|=n\left|\frac{\lambda}{n}e\right|\in E_{+}+...+E_{+}$ ($n$ times). Hence, $\left|\lambda e\right|\in E_{+}$, from where $\lambda e\in E$.\medskip

(ii)+(iv)$\Rightarrow$(v): Let $G$ be the set in (v); it is clear that $G$ is a linear subspace of $F\times F$. If $\left(e,f\right)\in G$, then $e-f\in E$, from where $\left|\left|e\right|-\left|f\right|\right|\le\left|e-f\right|\in E$, hence $\left|e\right|-\left|f\right|\in E$, and so $\left(\left|e\right|,\left|f\right|\right)\in G$.\medskip

(v)$\Rightarrow$(ii): First, $E\times\left\{0\right\}=G\cap \left(F\times\left\{0\right\}\right)$ is a sublattice of $F\times F$, hence a sublattice of $F\times\left\{0\right\}$, from where $E$ is a sublattice of $F$. Next, if $e\in E_{+}$, $f\in F_{+}$, then $\left(e,0\right), \left(f,f\right)\in G$, from where $\left(e\wedge f,0\right)=\left(e,0\right)\wedge \left(f,f\right)\in G$, thus $f\wedge e\in E$.
\end{proof}

The intersection of sublattices / solids / ideals is a sublattice / solid / ideal. The intersection of all ideals that contain $G\subset F$ will be denoted by $I\left(G\right)$; we call $G$ \emph{majorizing} if $I\left(G\right)=F$. A \emph{principal ideal} is an ideal of the form $I_{e}:=I\left(\left\{e\right\}\right)=\bigcup\limits_{\alpha\ge 0}\alpha\left[-e,e\right]$ -- the smallest ideal that contains $e\in F_{+}$. According to Krein-Kakutani Theorem (see \cite[Theorem 45.3]{zl}), if $F$ is Archimedean, for every $e$ there is a compact Hausdorff space $K_{e}$ such that $I_{e}$ is vector lattice isomorphic to a norm dense sublattice of $\Co\left(K_{e}\right)$; in this case we will call $K_{e}$ the \emph{Krein-Kakutani spectrum} of $e$. If $F=I_{e}$, we will call $e$ a \emph{strong unit} of $F$. If $I_{e}$ is isomorphic to $\Co\left(K_{e}\right)$, for every $e\in F_{+}$, we call $F$ \emph{uniformly complete}. Note that this condition is equivalent to the fact that every $I_{e}$ is a Banach space with respect to the Minkowski functional $\|\cdot\|_{e}$ of $\left[-e,e\right]$. We will call $L\subset F$ \emph{uniformly closed} if $L\cap I_{e}$ is $\|\cdot\|_{e}$-closed in $I_{e}$, for every $e\in F_{+}$.\medskip

Note that an ideal of a vector lattice is not an ideal of the underlying lattice (unless $F=\left\{0\right\}$). We will denote the set of the vector lattice ideals of $F$ by $\mathcal{I}_{F}$, and the set of the lattice ideals of $F_{+}$ by $\Jo_{F_{+}}$. Clearly, both of these sets are complete lattices with respect to the inclusion, since the intersection of any collection of ideals is an ideal; the maximum and minimum of $\mathcal{I}_{F}$ are $F$ and $\left\{0\right\}$, respectively, while for $\Jo_{F_{+}}$ these roles are played by $F_{+}$ and $\left\{0\right\}$.  Clearly, $E\mapsto E_{+}$ is a map from $\mathcal{I}_{F}$ into $\mathcal{J}_{F_{+}}$, which respects the intersection, and we will now investigate some of its properties.

\begin{proposition}\label{lid}If $E,H\in\Io_{F}$, then $E\vee_{\Io_{F}}H=E+H$ with $E_{+}\vee_{\mathcal{J}_{F_{+}}} H_{+}=E_{+}\vee H_{+}=E_{+}+H_{+}=\left(E+H\right)_{+}$. Moreover, if $\mathcal{J}\subset \mathcal{I}_{F}$, then $$\bigvee\nolimits_{\mathcal{I}_{F}} \mathcal{J}= \bigplus\mathcal{J}:=\left\{f_{1}+...+f_{n},~ f_{k}\in E_{k}\in \mathcal{J}\right\},$$
and $$\left(\bigvee\nolimits_{\mathcal{I}_{F}} \mathcal{J}\right)_{+}=\bigplus\limits_{J\in\mathcal{J}}J_{+}=\bigvee\limits_{J\in\mathcal{J}}J_{+}=\bigvee\nolimits_{\mathcal{J}_{F_{+}}}\left\{J_{+},~J\in\mathcal{J}\right\}.$$
\end{proposition}
\begin{proof}
First, since $E,H$ are solid, so is $E+H$ and $\left(E+H\right)_{+}=E_{+}+H_{+}$. Hence, according to Proposition \ref{id}, $E+H\in \Io_{F}$, and it is clear that $E\vee_{\Io_{F}}H=E+H$. Next, $E_{+}+H_{+}=\left(E+H\right)_{+}\in \Io_{F_{+}}$ with $e\vee h\le e+h\le 2\left(e\vee h\right)$, for any $e\in E_{+}$ and $h\in H_{+}$, yield $E_{+}\vee_{\mathcal{J}_{F_{+}}} H_{+}=E_{+}\vee H_{+}=E_{+}+H_{+}$.\medskip

It now follows that the claims are true for any finite collection of ideals. The general case now also follows because e.g. $\bigvee_{\mathcal{I}_{F}} \mathcal{J}$ and $\bigplus\mathcal{J}$ are both the union of the directed family of $\bigvee_{\mathcal{I}_{F}} \mathcal{J}'=\bigplus\mathcal{J}'$, over all finite $\mathcal{J}'\subset \mathcal{J}$.
\end{proof}

\begin{corollary}\label{lidd}
The map $E\mapsto E_{+}$ is a lattice homomorphism from $\mathcal{I}_{F}$ into $\mathcal{J}_{F_{+}}$, which preserves arbitrary joins and meets, and has a left inverse. Hence, $\mathcal{I}_{F}$ embeds as a regular sublattice of $\mathcal{J}_{F_{+}}$, and so is infinite-join distributive.
\end{corollary}

Recall that $e,f\in F$ are called \emph{disjoint} (denoted $e\bot f$) if $\left|e\right|\wedge \left|f\right|=0$. Note that disjoint $e,f$ are not disjoint in the underlying lattice of $F$ (unless $F=\left\{0\right\}$), since the latter has no minimal element. However, for $e,f\in F_{+}$ vector-lattice-disjointness in $F$ is equivalent to lattice-disjointness in $F_{+}$. We also have the following relation between the two concepts.

\begin{proposition}\label{prin}
\item[(i)] The correspondence $f\mapsto I_{f}$ is a homomorphism from $F_{+}$ into $\mathcal{I}_{F}$. In particular, it preserves disjointness.
\item[(ii)] If $G,H\subset F_{+}$, then $I\left(G\right)\cap I\left(H\right)=I\left(\left\{g\wedge h,~ g\in G,~ h\in H\right\}\right)$.
\end{proposition}
\begin{proof}
(i) follows from the easily established identities $I_{e\vee f}=I_{e}+I_{f}$ and $I_{e\wedge f}=I_{e}\cap I_{f}$. (ii) follows from infinite-join distributivity of $\mathcal{I}_{F}$, the fact that $I\left(G\right)=\bigvee_{\Io_{F}}\left\{I_{g}, g\in G\right\}$, and (i).
\end{proof}

We will discuss infinite-meet distributivity of $\mathcal{I}_{F}$ later in the article. An ideal of a vector lattice which contains the supremum of any of its subsets is called a \emph{band}. Intersection of any collection of bands is a band, and so the set $\Bo_{F}$ of all bands of $F$ is a complete lattice with respect to the inclusion. For any $G\subset F$ its disjoint complement $G^{d}$ is a band, and conversely, if $F$ is Archimedean, then $E\subset F$ is a band if and only if $E=E^{dd}$. In fact, $G^{dd}$ is the band generated by $G\subset F$ (see \cite[Theorem 1.39]{ab}). If $E\subset F$ is a sublattice and $G\subset E$, then $G^{d}_{E}$ denotes the disjoint complement of $G$ with respect to the vector lattice structure of $E$, i.e. $G^{d}_{E}=G^{d}\cap E$. Note that $G^{dd}_{E}:=\left(G^{d}_{E}\right)^{d}_{E}$ is not necessarily equal to $G^{dd}\cap E$ (see \cite[Proposition 2.3 and Corollary 3.14]{erz} for more details). It turns out that the disjoint complement is the pseudo-complementation in the lattice $\Io_{F}$.

\begin{corollary}\label{idis}For $E,H\in \mathcal{I}_{F}$ the following conditions are equivalent:
\item[(i)] $E,H$ are disjoint as elements of $\mathcal{I}_{F}$;
\item[(ii)] $E\cap H=\left\{0\right\}$;
\item[(iii)] $e\bot h$, for all $e\in E$, $h\in H$;
\item[(iv)] $e\wedge h=0$, for all $e\in E_{+}$, $h\in H_{+}$.\medskip

Furthermore, for every $H\in \mathcal{I}_{F}$ we have $H^{d}=H^{*}_{\mathcal{I}_{F}}$.
\end{corollary}
\begin{proof}
(i)$\Leftrightarrow$(ii) follows from the fact that $\wedge_{\mathcal{I}_{F}}=\cap$.

If $E\cap H=\left\{0\right\}$ and $e\in E_{+}$, $h\in H_{+}$, then $e\wedge h\in E\cap H$, from where $e\wedge h=0$. If $e\wedge h=0$, for all $e\in E_{+}$, $h\in H_{+}$, and $e\in E$, $h\in H$, then $\left|e\right|\in E_{+}$, $\left|h\right|\in H_{+}$, and so $\left|e\right|\wedge \left|h\right|=0$, hence $e\bot h$. Finally, if $e\bot h$, for all $e\in E$, $h\in H$, and $g\in E\cap H$, then $g\bot g$, and so $g=0$. The last claim follows from the equivalences.
\end{proof}

An ideal $H\subset F$ is called a \emph{projection band} if $F=H+H^{d}$ (and so there is a projection $P:F\to H$, which is a homomorphism). It is well-known (see \cite[Theorem 1.41]{ab}) that $H$ is a projection band if and only if $\bigvee\left(H\cap\left[0,f\right]\right)$ exists in $H$, for every $f\in F_{+}$. It is easy to see that every projection band is in fact a band, and $F$ is said to have the \emph{projection property (PP)} if the converse is also true. We will also say that $F$ has the \emph{principal projection property (PPP)} if $\left\{e\right\}^{dd}$ is a projection band, for every $e\in F$. This is equivalent to the fact that $\left\{e\right\}^{d}$ is a projection band, for every $e\in F$. Similarly, we will say that $F$ has $\sigma$\emph{-PP} if every countably generated band is a projection band (which is equivalent to the fact that the disjoint complement of any sequence is a projection band).

\begin{corollary}\label{com}
If $F$ is Archimedean, then bands are the pseudo-complemented elements of $\mathcal{I}_{F}$; they form a Boolean algebra, which is a meet-closed subset of $\mathcal{I}_{F}$. Projection bands are the complemented elements of $\mathcal{I}_{F}$; they form a Boolean algebra, which is a sublattice of $\mathcal{I}_{F}$.
\end{corollary}

\begin{remark}
It is also well-known that $H\in \mathcal{I}_{F}$ is prime if and only if it is join-irreducible element of $\mathcal{I}_{F}$ (in fact, in \cite{zl} it is taken as the definition). For a recent study of the properties of prime ideals of a vector lattice see e.g. \cite{kr}. Note that an ideal is maximal if and only if it is simultaneously prime and uniformly closed. Since both maximality and primeness are ``horizontal'' properties, it is natural to ask whether the same is true about uniform closeness, i.e. is it possible to find an intrinsic description of the subset of $\mathcal{I}_{F}$, which consists of uniformly closed ideals. Since these are precisely the ideals which produce Archimedean quotients (see \cite[Theorem 60.2]{zl}), this question is equivalent to the following:
\end{remark}

\begin{question}
Is it possible to ascertain whether $F$ is Archimedean in terms of $\Io_{F}$?
\end{question}

\begin{corollary}\label{rel}
\item[(i)] If $E$ and $J$ are ideals in $F$, then $\left(E\cap J\right)^{d}_{E}=E\cap J^{d}$, $\left(E\cap J\right)^{dd}_{E}=E\cap J^{dd}$ and $E^{dd}\cap J^{dd}=\left(E\cap J\right)^{dd}$.
\item[(ii)] If $G,H\subset F_{+}$, then $G^{dd}\cap I\left(H\right)=\left\{g\wedge h,~ g\in G,~ h\in H\right\}^{dd}_{I\left(H\right)}$, and\linebreak $G^{dd}\cap H^{dd}=\left\{g\wedge h,~ g\in G,~ h\in H\right\}^{dd}$. Also, $\left\{g\wedge h,~ g\in G,~ h\in H\right\}^{d}=\left(G^{dd}\cap H^{dd}\right)^{d}=G^{d}\vee_{\Bo_{F}} H^{d}$.
\end{corollary}
\begin{proof}
(i) follows from Glivenko Theorem. (ii) follows from combining (i) with part (ii) of Proposition \ref{prin}.
\end{proof}

It follows that $F$ has PP if and only if in $\mathcal{I}_{F}$ pseudo-complements are complemented. PPP ($\sigma$-PP) is equivalent to the fact that pseudo-complements of the principal (countably generated) ideals are complemented. It can be easily deduced from Corollary \ref{rel} that all these properties are inherited by ideals.

Recall that a sublattice $E\subset F$ is called \emph{order dense} if $E\cap \left(0,f\right]\ne\varnothing$, for every $f\in F_{+}\backslash \left\{0\right\}$. If $F$ is Archimedean, this condition is equivalent to the fact that $f=\bigvee\left(E\cap \left[0,f\right]\right)$, for every $f\in F_{+}$ (see \cite[Theorem 1.34]{ab}).

\section{Sublattices of $\Co\left(X\right)$}\label{s}

Everywhere in this section $X$ is a Tychonoff topological space. We denote the space of all continuous functions on $X$ by $\Co\left(X\right)$. This space is a vector lattice with respect to the pointwise operations, and so it is a sublattice of the vector lattice $\Fo\left(X\right)$ of all real-valued functions on $X$. It is easy to see that these lattices are Archimedean. The interval $\left[f,g\right]$, where $f,g\in \Fo\left(X\right)$ will be meant in $\Co\left(X\right)$ as a sublattice of $\Fo\left(X\right)$, i.e. all continuous functions between (not necessarily continuous) $f$ and $g$.

We will denote the function which is identically $0$ on $X$ by $\0$, while $\1_{A}$ is the indicator of $A\subset X$, i.e. the function whose value is $1$ on $A$ and $0$ on $X\backslash A$; we also put $\1=\1_{X}$.

The space $\Co\left(X\right)$ is by default endowed with the compact-open topology. Recall that a sublattice $E\subset \Co\left(X\right)$ is dense if and only if it \emph{strictly separates points}, i.e. if for any distinct $x,y\in X$ there is $f\in E$ such that $f\left(x\right)=1$ and $f\left(y\right)=0$ (see e.g. \cite{et2}). Let $\Co_{b}\left(X\right)$ stand for the Banach lattice of all bounded continuous functions on $X$ endowed with the supremum norm $\|\cdot\|$. If $f:X\to\R$ and $A\subset X$, define $\|f\|_{A}=\bigvee\limits_{x\in A}\left|f\left(x\right)\right|$.

\subsection{Urysohn lemmas}

By definition of the Tychonoff space, for every open $U\subset X$ and $x\in U$, there is $f\in \left[\1_{\left\{x\right\}},\1_{U}\right]$. Moreover, if $K\subset U$ is compact, or if $X$ is normal and $K$ is closed, Tietze-Urysohn theorem guarantees that there is $f\in \left[\1_{K},\1_{U}\right]$ (see \cite[theorems 2.1.8 and 3.1.7]{engelking}). Also, if $K=U$, then $\1_{K}\in\Co_{b}\left(X\right)$.

\begin{proposition}[Norm-dense Urysohn lemma]\label{dul}
Let $E$ be a dense sublattice of $\Co_{b}\left(X\right)$, let $U$ be an open neighborhood of a compact set $K$ and let $f\in E_{+}$. Then, there is $e\in E\cap \left[\0,f\right]$ which coincides with $f$ on $K$ and vanishes outside of $U$ (or vice versa). If additionally $X$ is normal, or $K=U$, then $K$ can be assumed to be merely closed.
\end{proposition}
\begin{proof}
Without loss of generality we may assume that $\0\le f\le \1$. There is $h\in \Co_{b}\left(X\right)$ which vanishes outside $U$ and equals $3$ on $K$. Since $E$ is dense, there is $g'\in E$ such that $\|h-g'\|\le 1$. Then, $\left.g'\right|_{K}\ge 2$ and $\left.g'\right|_{K\backslash U}\le 1$. Similarly, there is $g''\in E$ such that $\left.g''\right|_{K}\le 1$ and $\left.g''\right|_{K\backslash U}\ge 2$. Take $g=\left(g'-g''\right)^{+}$; we have $\left.g\right|_{K}\ge 1$ and $\left.g\right|_{K\backslash U}\equiv 0$, and so $e=f\wedge g$ satisfies the requirements.
\end{proof}

Since the set of the restrictions of elements of a dense sublattice of $\Co\left(X\right)$ to a compact set $K\subset X$ forms a dense sublattice of $\Co\left(K\right)=\Co_{b}\left(K\right)$, we get the following result.

\begin{corollary}\label{ssul}
Let $E$ be a dense sublattice of $\Co\left(X\right)$, let $K,L\subset X$ be compact and disjoint, and let $f\in E$. Then, there is $e\in E$ which coincides with $f$ on $K$ and vanishes on $L$.
\end{corollary}

Let us consider a similar characterization for order dense sublattices of $\Co\left(X\right)$. It was proven in \cite[Proposition 5.3]{erz} that $E$ is order dense if and only if $E\cap\left(\0,\1_{U}\right]\ne \varnothing$, for every nonempty open $U\subset X$. Note that $\Co_{b}\left(X\right)$ is order dense in $\Co\left(X\right)$.

\begin{proposition}[Order dense Urysohn lemma]\label{od}Let $E$ be a sublattice of $\Co\left(X\right)$. Then:
\item[(i)] If $E$ is order dense, then for every $f\in E_{+}$ and every nonempty open $U$ there is an nonempty open $V\subset U$ and $e\in E\cap \left[\0,f\right]$ which vanishes outside of $U$ and coincides with $f$ on $V$.
\item[(ii)] If $\1\in E$, then $E$ is order dense if and only if for every nonempty open $U$ there is an nonempty open $V\subset U$ such that $E\cap \left[\1_{\overline{V}},\1_{U}\right]\ne\varnothing$.
\end{proposition}
\begin{proof}
(i): Let $f\in E_{+}$ and let $U\subset X$ be open and nonempty. By restricting $U$ if necessary, we may assume that $f$ is bounded on $U$. Since $E$ is order dense, there is $g\in E\cap\left(\0,\1_{U}\right]$. In particular, $g$ vanishes outside of $U$. There is $\delta>0$ such that $V=g^{-1}\left(\delta,+\8\right)\ne\varnothing$. Then, $e=\frac{\|f\|_{U}}{\delta}g\wedge f\in E$ vanishes outside of $U$ and coincides with $f$ on $V$.\medskip

Sufficiency in (ii) immediately follows from the characterization of order density quoted above, while necessity follows from (i) for $f=\1$.
\end{proof}

We will call $X$ \emph{almost locally compact} if the set of all points of $X$ that have a compact neighborhood is dense.

\begin{proposition}
Let $E$ be a sublattice of $\Co\left(X\right)$. If $E\subset \Co_{b}\left(X\right)$ or if $X$ is almost locally compact, then $E$ is order dense if and only if for every nonempty open $U$ there is $f\in E\backslash \left\{\0\right\}$ which vanishes outside of $U$.
\end{proposition}
\begin{proof}
Necessity is clear, while sufficiency in the case of $E\subset \Co_{b}\left(X\right)$ easily follows from the characterization of order density quoted before Proposition \ref{od}.

Assume that $X$ is almost locally compact. Let $\Co_{00}\left(X\right)$ be the set of all continuous functions on $X$ which vanish outside of a compact set. Clearly, $E\cap \Co_{00}\left(X\right)\subset \Co_{00}\left(X\right)\subset \Co_{b}\left(X\right)$. Any nonempty open $U$ contains a nonempty open relatively compact $V$, and there is $f\in E\backslash \left\{\0\right\}$ which vanishes outside of $V$, and so is an element of $E\cap \Co_{00}\left(X\right)$. Hence, $E\cap \Co_{00}\left(X\right)$ is order dense, from where $E$ is order dense as well.
\end{proof}

\subsection{Ideals in sublattices of $\Co\left(X\right)$}

For any $f:X\to\R$ define $\ker f=f^{-1}\left(0\right)$ and $\supp f=X\backslash \ker f$. A set is called a \emph{cozero set} if it is of the form $\supp f$, where $f\in \Co\left(X\right)$. For $G\subset \Co\left(X\right)$ define $\ker G=\bigcap\limits_{g\in G}\ker g$, which is closed, and $\supp G=\bigcup\limits_{g\in G}\supp g=X\backslash \ker G$, which is an open set. Note that if $E\subset\Co\left(X\right)$ is dense, then $\supp E=X$.

\begin{lemma}\label{ssup}
Let $E$ be a sublattice of $\Co\left(X\right)$ and let $K\subset \supp E$ be compact. Then, there is $e\in E$ such that $e\ge \1_{K}$.
\end{lemma}
\begin{proof}
For every $x\in K$ there is $e_{x}\in E$ such that $e_{x}\left(x\right)>1$, and so $e_{x}^{-1}\left(1,+\8\right)$ is an open neighborhood of $x$. As $K$ is compact, there are $e_{x_{1}},...,e_{x_{n}}\in E$, such that $K\subset \bigcup\limits_{k=1}^{n}e_{x_{k}}^{-1}\left(1,+\8\right)= e^{-1}\left(1,+\8\right)$, where $e=\bigvee\limits_{k=1}^{n}\left|e_{x_{k}}\right|\in E$.
\end{proof}

If $E$ is a sublattice of $\Co\left(X\right)$, and $A\subset X$, define $E\left(A\right)=\left\{f\in E,~\supp f\subset A\right\}$. Clearly, $E\left(A\right)$ is an ideal in $E$. Since elements of $E$ are continuous, it follows that $E\left(A\right)=E\left(\Int A\right)$. If $B\subset X$, then $E\left(A\cap B\right)=E\left(A\right)\cap E\left(B\right)$, and if $A\subset B$, then $E\left(A\right)\subset E\left(B\right)$. In the case when $E=\Co\left(X\right)$, we will use the notation $\Co\left(X;A\right)$ for $E\left(A\right)$, so that for a general $E$ we have $E\left(A\right)=\Co\left(X;A\right)\cap E$. Recall that any closed ideal in $\Co\left(X\right)$ is of the form $\Co\left(X;U\right)$, for a unique open $U$, and moreover, $\overline{I\left(G\right)}=\Co\left(X;\supp G\right)$, for any $G\subset \Co\left(X\right)$ (see e.g. \cite{et2}; note that $\supp I\left(G\right)\supset\supp G$, but since $I\left(G\right)\subset \Co\left(X;\supp G\right)$, we have $\supp I\left(G\right)=\supp G$). Let us extend this to the ideals in sublattices of $\Co\left(X\right)$. In what follows these sublattices are endowed with the compact-open topology inherited from $\Co\left(X\right)$.

\begin{proposition}\label{ds}Let $E\subset\Co\left(X\right)$ be a dense sublattice.
\item[(i)] If $J$ is an ideal in $E$, then $\overline{J}^{\Co\left(X\right)}=\Co\left(X;\supp J\right)$ and $\overline{J}^{E}=E\left(\supp J\right)$.
\item[(ii)] Every closed ideal of $E$ is of the form $E\left(U\right)$, for some open $U\subset X$.
\item[(iii)] $U\mapsto E\left(U\right)$ is a meet-preserving map from the lattice of open subsets of $X$ onto the lattice of the closed ideals of $E$.
\item[(iv)] If $J$ is an ideal in $E$ and $K$ is a compact subset of $\supp J$, then $E\left(K\right)\subset J$. In particular, if $\supp J$ is compact (and so clopen), then $J=E\left(\supp J\right)$ and is a principal ideal.
\end{proposition}
\begin{proof}
(i): Let $U=\supp J$. Clearly, $\overline{J}^{\Co\left(X\right)}\subset \Co\left(X;U\right)$. If $x,y\in U$ are distinct, take $f\in J_{+}$ with $f\left(x\right)>0$. Since $E$ is dense, there is $e\in E_{+}$ with $e\left(x\right)>0$ and $e\left(y\right)=0$. Then, $e\wedge f\in J$ strictly separates $x$ from $y$. Since $x,y$ were arbitrary, $\overline{J}^{\Co\left(X\right)}=\Co\left(X;U\right)$. Thus, $\overline{J}^{E}=\overline{J}^{\Co\left(X\right)}\cap E=\Co\left(X;U\right)\cap E=E\left(U\right)$.\medskip

(ii) follows from (i), since if $J$ is closed in $E$, then $J=\overline{J}^{E}=E\left(\supp J\right)$. (iii) follows from (ii).\medskip

(iv): From Lemma \ref{ssup}, one can find $e\in J$ such that $e\ge\1_{K}$. Then, if $f\in E$ is such that $\supp f\subset K$, it follows that $\left|f\right|\le \|f\|e$, and so $f\in J$. The last claim follows from combining the first claim with Lemma \ref{ssup}.
\end{proof}

Part (iii) of Proposition \ref{ds} motivates the following question.

\begin{question}\label{qu}
For which $E$ is it true that $U\mapsto E\left(U\right)$ is join-preserving into $\Io_{E}$, i.e. $E\left(U\right)+E\left(V\right)=E\left(U\cup V\right)$, for any open $U,V\subset X$?
\end{question}

We will later consider a class of $E$'s for which this correspondence is join-preserving into the lattice of closed ideals in $E$. Consider the case when $E=\Co\left(X\right)$.

\begin{proposition}\label{normal}
$X$ is normal if and only if $\Co\left(X;U\right)+\Co\left(X;V\right)=\Co\left(X;U\cup V\right)$, for any open $U,V\subset X$, and if and only if $\Co\left(X;U\right)+\Co\left(X;V\right)=\Co\left(X\right)$, for any open $U,V\subset X$ such that $U\cup V=X$.
\end{proposition}
\begin{proof}
Let $X$ be normal, let $U,V\subset X$ be open, and let $A:=X\backslash U$ and $B:=X\backslash V$, which are closed. Let $f\in \Co\left(X;U\cup V\right)$. Define $h:A\cup B\to \R$ by $h\left(x\right)=f\left(x\right)$, if $x\in A$, and $h\left(x\right)=0$, if $x\in B$. Note that $f$ vanishes on $A\cap B$ and so $h$ is well-defined and continuous. Hence, by Tietze-Urysohn theorem it can be extended to a continuous function on $X$ (which we also denote $h$). It follows that $h\in \Co\left(X;V\right)$, while $f-h\in\Co\left(X;U\right)$.

Conversely, assume that $A,B\subset X$ are closed and disjoint. Let $U:=X\backslash A$ and $V:=X\backslash B$, which are open and $U\cup V=X$. Then, $\1\in\Co\left(X\right)=\Co\left(X;U\right)+\Co\left(X;V\right)$, hence there is $f\ge\0$ which vanishes on $A$ and $g\ge\0$ that vanishes on $B$ such that $f+g=\1$. Then, $f^{-1}\left[0,\frac{1}{2}\right)$ and $g^{-1}\left[0,\frac{1}{2}\right)$ are disjoint neighborhoods of $A$ and $B$, respectively. Thus, $X$ is normal.
\end{proof}

\begin{remark}\label{njo}Analogously, $X$ is normal if and only if $\Co_{b}\left(X;U\right)+\Co_{b}\left(X;V\right)=\Co_{b}\left(X;U\cup V\right)$, for any open $U,V\subset X$, and if and only if $\Co_{b}\left(X;U\right)+\Co_{b}\left(X;V\right)=\Co_{b}\left(X\right)$, for any open $U,V\subset X$ such that $U\cup V=X$. However, it is not always true that if $E$ is a dense sublattice of $\Co\left(K\right)$, where $K$ is compact and Hausdorff and $U,V\subset K$ are open, then $E\left(U\cup V\right)=E\left(U\right)+E\left(V\right)$. Let $E$ be a subset of $\Co\left[-1,1\right]$ that consists of functions which are even in a neighborhood of $0$. It is easy to verify that $E$ is a dense sublattice of $\Co\left[-1,1\right]$, but $f$ defined by $f\left(t\right)=\left|t\right|$ belongs to $E\left(\left[-1,0\right)\cup \left(0,1\right]\right)$, but cannot be decomposed into a sum of elements of $E\left(\left[-1,0\right)\right)$ and $E\left(\left(0,1\right]\right)$.
\end{remark}

\subsection{Urysohn sublattices}

We will call a sublattice $E\subset \Co\left(X\right)$ an \emph{Urysohn sublattice} if for every open neighborhood $U$ of $x\in X$ there is $f\in E\left(U\right)$ with $f\left(x\right)=1$ (and in particular $E$ is dense in $\Co\left(X\right)$). In other words, $E$ is Urysohn if $\supp E\left(U\right)=U$, for any open set $U\subset X$. Clearly, $\Co\left(X\right)$ itself is Urysohn. According to the norm-dense Urysohn Lemma, every norm-dense sublattice of $\Co_{b}\left(X\right)$ is Urysohn. Using the same method one can show that if $X$ is locally compact, then any dense sublattice of $\Co_{0}\left(X\right)$ (e.g. $\Co_{00}\left(X\right)$) is Urysohn. The next result is proven by applying Lemma \ref{ssup} to $E\left(U\right)$.

\begin{lemma}\label{ssuup}
Let $E$ be an Urysohn sublattice of $\Co\left(X\right)$ and let $U\subset X$ be an open neighborhood of a compact $K\subset X$. Then, there is $e\in E\left(U\right)$ such that $e\ge \1_{K}$.
\end{lemma}

For Urysohn sublattice $E$ we have that  $\supp E\left(A\right)=\Int A$, for any $A\subset X$, and so $E\left(A\right)\subset E\left(B\right)$ $\Leftrightarrow$ $\Int A\subset \Int B$, for any $A,B\subset X$. Therefore, the correspondence $U\mapsto E\left(U\right)$ is an order isomorphism between the set of open subsets of $X$ and the set of the closed ideals in $E$. In particular, for any open $U,V\subset X$ we have that $\supp\left(E\left(U\right)+E\left(V\right)\right)=U\cup V$. It follows that $E\left(U\cup V\right)= \overline{E\left(U\right)+E\left(V\right)}$. This isomorphism may still not be join-preserving, when considered as a map into $\Io_{E}$, because it is possible that $E\left(U\cup V\right)\ne E\left(U\right)+E\left(V\right)$, for some open $U,V\subset X$ (see Remark \ref{njo}). Hence, the Question \ref{qu} for Urysohn lattices can be restated as follows.

\begin{question}\label{qu}
For which Urysohn sublattices $E$ is the sum of two closed ideals closed?
\end{question}

According to Proposition \ref{normal}, the sum of closed ideals in $\Co\left(X\right)$ is closed if and only if $X$ is normal. Let us investigate a weaker property.

\begin{proposition}\label{dech}
Let $E\subset \Co\left(X\right)$ be a sublattice with $\supp E=X$. Let $U,V\subset X$ be open. If $E=E\left(U\right)+E\left(V\right)$, then $U\cup V=X$. The converse holds in each of the following cases:
\item[(i)] $E$ is dense in $\Co\left(X\right)$ and $U,V$ are co-compact (i.e. have compact complements);
\item[(ii)] $E$ is Urysohn and either $U$ or $V$ is co-compact;
\item[(iii)] $X$ is normal and either $E=\Co\left(X\right)$ or $E$ is dense in $\Co_{b}\left(X\right)$;
\item[(iv)] $U\cap V=\varnothing$ (and so $U,V$ are clopen) and either $E=\Co\left(X\right)$ or $E$ is dense in $\Co_{b}\left(X\right)$.
\end{proposition}
\begin{proof}
If $x\notin U\cup V$, then $E\left(U\right),E\left(V\right)\subset E\left(X\backslash \left\{x\right\}\right)$, from where $E\left(U\right)+E\left(V\right)\subset E\left(X\backslash \left\{x\right\}\right) \subsetneq E$.\medskip

Now assume that $U\cup V=X$. We will prove that $E=E\left(U\right)+E\left(V\right)$ in each of the cases.\medskip

(i): If $K=X\backslash U$ and $L=X\backslash V$ are compact, then $K\cap L=\varnothing$, and so from Corollary \ref{ssul}, there is $e\in E$ that vanishes on $K$ (and so $e\in E\left(U\right)$), and coincides with $f$ on $L$ (and so $f-e\in E\left(V\right)$).\medskip

(ii): If $L=X\backslash V$ is compact, then $U$ is an open neighborhood of $L$. According to Lemma \ref{ssuup} there is $e\in E\left(U\right)$ such that $e\ge\1_{L}$. If $f\in E_{+}$, then $g=\|f\|_{L}e\wedge f\in E\left(U\right)$, and coincides with $f$ on $L$ (and so $f-g\in E\left(V\right)$).\medskip

(iii) and (iv) follow from Proposition \ref{normal} and the norm-dense Urysohn Lemma.
\end{proof}

\begin{remark}\label{urur}
A reformulation of (ii) is that if $E$ is an Urysohn sublattice of $\Co\left(X\right)$ and $U\subset X$ is an open neighborhood of a compact $K\subset X$, for every $f\in E_{+}$ there is $e\in E\cap \left[\0,f\right]$ which coincides with $f$ on $K$ and vanishes outside of $U$ (or vice versa). This is yet another version of the Urysohn lemma for sublattices.
\end{remark}

\begin{corollary}\label{dsi}
If $E\subset \Co\left(X\right)$ is an Urysohn sublattice, $H\subset E$ is an ideal such that $\supp H$ is locally compact, $K\subset \supp H$ is compact and $f\in E$, there is $h\in H\cap \left[-f,f\right]$ which coincides with $f$ on $K$.
\end{corollary}
\begin{proof}
Let $U$ be an open neighborhood of $K$ such that $\overline{U}$ is a compact subset of $\supp H$. Let $e\in E$ be produced by Remark \ref{urur} for $E,f,K,U$. Then, $\overline{\supp e}\subset \overline{U}$ is a compact subset of $\supp H$. Hence, $e\in H$, by virtue of part (iv) of Proposition \ref{ds}.
\end{proof}

It is clear that if $J$ is a countably generated ideal in a sublattice $E$ of $\Co\left(X\right)$, then $\supp J$ is a cozero set. If $f\in\Co_{0}\left(X\right)$, then $\supp f$ is a locally compact $\sigma$-compact set. Using Urysohn lemma one can get a partial converse to these two observations.

\begin{lemma}\label{ds2}Let $E\subset\Co\left(X\right)$ be a Urysohn sublattice and let $J$ be an ideal in $E$. If $\supp J$ is a locally compact $\sigma$-compact set, there is a countable $G\subset J$, such that $\supp J=\supp G$.
\end{lemma}
\begin{proof}
There is a sequence $\left(V_{n}\right)_{n\in\N}$ of open relatively compact sets that add up to $\supp J$ and such that $\overline{V_{n}}\subset V_{n+1}$, for every $n\in\N$. For every $n\in\N$, since $V_{n+1}$ is an open neighborhood of a compact set $\overline{V_{n}}$, from Lemma \ref{ssuup} there is $f_{n}\in E\left(V_{n+1}\right)$, which does not vanish on $\overline{V_{n}}$; since $\overline{\supp f_{n}}\subset \overline{V_{n+1}}$, it follows from part (iv) of Proposition \ref{ds} that $f_{n}\in J$. Letting $G=\left\{f_{n}\right\}_{n\in\N}$ we have $\supp G=\supp J$.
\end{proof}

It is easy to see that $\left\{f\right\}^{d}=\Co\left(X;\ker f\right)$, for any $f\in \Co\left(X\right)$, and if $G\subset \Co\left(X\right)$, then $G^{d}=\bigcap\limits_{g\in G}\left\{g\right\}^{d} =\Co\left(X;\ker G\right)$. Hence, if $E$ is a sublattice of $\Co\left(X\right)$, we have $G^{d}_{E}=E\left(\ker G\right)$. If $E$ is an Urysohn sublattice, and $A\subset X$, then $E\left(A\right)_{E}^{d}=E\left(X\backslash \Int A\right)=E\left(X\backslash \overline{\Int A}\right)$, from where $$E\left(A\right)_{E}^{dd}=E\left(X\backslash \Int \left(X\backslash \Int A\right)\right)=E\left(\overline{\Int A}\right),$$  as well as  $G_{E}^{dd}= E\left(\overline{\supp G}\right)$, for $G\subset E$.

\begin{proposition}\label{ur}Let $E$ be a Urysohn sublattice of $\Co\left(X\right)$, and $A\subset X$. Then:
\item[(i)] $E\left(A\right)$ is a band in $E$ if and only if $\Int A=\Int \overline{\Int A}$, i.e. $\Int A$ is regularly open.
\item[(ii)] If $E\left(A\right)$ is a projection band in $E$, then $\Int A$ is clopen. The converse holds if $\Int A$ is compact or co-compact or $E=\Co\left(X\right)$, or is a dense sublattice of $\Co_{b}\left(X\right)$.
\end{proposition}
\begin{proof}
(i) follows from $E\left(A\right)_{E}^{dd}=E\left(\overline{\Int A}\right)$, and the fact that for Urysohn sublattices $E\left(A\right)= E\left(B\right)$ $\Leftrightarrow$ $\Int A=\Int B$, for any $A,B\subset X$.

(ii): $E\left(A\right)$ is a projection band in $E$ if and only if $E=E\left(A\right)+E\left(A\right)_{E}^{d}=E\left(\Int A\right)+E\left(X\backslash \overline{\Int A}\right)$, which according to Proposition \ref{dech} implies $X=\Int A\cup X\backslash \overline{\Int A}$, or clopeness of $\Int A$. From the same proposition $E=E\left(\Int A\right)+E\left(X\backslash \overline{\Int A}\right)$ also follows from clopeness of $\Int A$ under the declared additional assumptions.
\end{proof}

\subsection{Miscellaneous results about sublattices of $\Co\left(X\right)$}

The following result shows that usually $\Co\left(X\right)$ has a large supply of non-closed ideals.

\begin{proposition}\label{pi}
Let $f\in\Co\left(X\right)$. If $I_{f}$ is closed, then $\supp f$ is clopen. Conversely, if $\supp f$ is compact, then $I_{f}=\Co\left(X;\supp f\right)$ is a projection band (and so closed).
\end{proposition}
\begin{proof}
The converse clause follows immediately from part (iv) of Proposition \ref{ds}. Let us prove the main claim. We have that $\overline{I_{f}}=\Co\left(X;U\right)$, where $U=\supp f$, which is open. Assume that $\Co\left(X;U\right)\subset I_{f}$. Since $\sqrt{\left|f\right|}\in \Co\left(X;U\right)$ there is $s>0$ such that $\sqrt{\left|f\right|}\le s\left|f\right|$. Hence, if $x\in U$, then $s\sqrt{\left|f\left(x\right)\right|}\ge 1$, and so $U=\left|f\right|^{-1}\left[\frac{1}{s^{2}},+\8\right)$ is closed.
\end{proof}

We proceed with a characterization of the order dense ideals in sublattices of $\Co\left(X\right)$.

\begin{proposition}\label{rod}Let $E\subset\Co\left(X\right)$ be a sublattice. Then:
\item[(i)] If $H$ is an ideal in $E$ such that $\supp H$ is dense in $\supp E$, then $H$ is order dense in $E$.
\item[(ii)] If $E$ is regular, and $H$ is an order dense ideal in $E$, then $\supp H$ is dense in $\supp E$.
\item[(iii)] If $E$ is not regular, and either $E\subset \Co_{b}\left(X\right)$ or $X$ is almost locally compact, then there is an order dense ideal $H$ of $E$ such that $\supp H$ is not dense in $\supp E$.
\end{proposition}
\begin{proof}
(i): Assume that $\supp H$ is dense in $\supp E$, and let $e\in E_{+}\backslash\left\{\0\right\}$. Then, $\supp e\cap \supp H\ne\varnothing$, and so there is $h\in H_{+}$ such that $\supp e\cap \supp h\ne\varnothing$. Then $g=e\wedge h\in\left(\0,e\right]\cap H$, and so $H$ is order dense in $E$.\medskip

It was proven in \cite[Proposition 5.3]{erz} that $E$ is regular if and only if $\bigwedge_{E}\left\{f\in E,~f\ge \1_{U}\right\}\ne\0$ (this includes this set being empty, or having no infimum), for every nonempty open $U\subset X$.\medskip

(ii): Assume that $\supp H$ is not dense in $\supp E$, so that $U:=\supp E\backslash \overline{\supp H}\ne\varnothing$. Let $f\in E_{+}$ be such that $\supp f\cap U\ne\varnothing$. By scaling $f$ if needed we may assume that $f\ge \1_{V}$, where $V$ is an nonempty open subset of $\supp f\cap U$. Since $H$ is order dense, it follows that $\bigvee_{E} \left(f\wedge H\right)=f$, or equivalently, $\bigwedge_{E} \left(f-H\right)^{+}=\0$. On the other hand, since every element of $H$ vanishes on $U$ (and hence on $V$), it follows that $\left(f-H\right)^{+}\subset \left\{g\in E,~g\ge \1_{V}\right\}$, and so by the result quoted above, the infimum cannot be $\0$. Thus, we have reached a contradiction.\medskip

(iii): If $E$ is not regular, there is a nonempty open $U\subset X$ such that $\bigwedge_{E}\left\{g\in E,~g\ge \1_{U}\right\}=\0$; in particular the set is nonempty, and so $U\subset \supp E$. If $E\subset \Co_{b}\left(X\right)$, every element of $E$ is bounded on $U$. Let us show that $H:=E\left(X\backslash \overline{U}\right)$ is order dense in $E$. Let $f\in E_{+}$; we may assume that $\|f\|\le 1$. Therefore, if $g\in E$ is such that $g\ge \1_{U}$, then $\left(f-g\right)^{+}$ vanishes on $U$, hence on $\overline{U}$, and so $\left(f-g\right)^{+}\in H$. It then follows that $\bigwedge_{E} \left(f-H\right)^{+}\le \bigwedge_{E} \left\{\left(f-\left(f-g\right)^{+}\right)^{+},~g\in E,~g\ge \1_{U}\right\}=\bigwedge_{E} \left\{f\wedge g,~g\in E,~g\ge \1_{U}\right\}=\0$.\medskip

If $X$ is almost locally compact, there is an nonempty open $V\subset U$, such that $\overline{V}$ is compact. Then, every element of $E$ is bounded on $V$, and using the same arguments as in the first case, one can then show that $H:=E\left(X\backslash \overline{V}\right)$ is order dense in $E$.
\end{proof}

\begin{question}
Can the additional assumptions be removed from part (iii)? In other words, does the converse to (ii) hold?
\end{question}

The following lemma is inspired by \cite{abr}, \cite{ber} and \cite{vg} and will come very handy later. We will call it the ABVG Lemma.

\begin{lemma}\label{bern}
Let $K$ be compact Hausdorff, let $F\subset\Co\left(K\right)$ be a dense sublattice and let $H\subset F$ be the ideal generated by an increasing sequence $\left(h_{n}\right)_{n\in\N}$. For every $f\in F_{+}$, there is a sequence $\left(f_{n}\right)_{n\in\N}\subset H_{+}$ with $g_{n}:=\sum\limits_{k=1}^{n}f_{k}$ such that:
\begin{itemize}
\item $f_{n}\bot f_{m}$, for $\left|m-n\right|>1$;
\item $ f-g_{n}\bot f_{m}$, for every $m<n$;
\item $0\le f-g_{n}\le \left(f-h_{n}\right)^{+}+\frac{1}{n}\1$, for every $n\in\N$;
\item If $g\ge g_{n}$, for every $n\in\N$, then $\left(f-g\right)^{+}\bot H$.
\end{itemize}
\end{lemma}
\begin{proof}
Define $K_{n}=h_{n}^{-1}\left[\frac{1}{n},+\8\right)$, so that $K_{n}\subset h_{n}^{-1}\left(\frac{1}{n+1},+\8\right)\subset h_{n+1}^{-1}\left(\frac{1}{n+1},+\8\right)\subset \Int K_{n+1}$ and $\bigcup\limits_{n\in\N} K_{n}=\supp~ H$. Also, let $K_{0}=\varnothing$ and define $g_{0}=\0$. From Corollary \ref{dsi}, for every $n\in\N$ there is $f_{n}\in H\cap \left[\0,f-g_{n-1}\right]$, which coincides with $f-g_{n-1}$ on $K_{n}$ and vanishes outside of $\Int K_{n+1}$. It then follows that $g_{n}$ coincides with $f$ on $K_{n}$, and so $\supp f_{n}\subset K_{n+1}\backslash K_{n-1}$. Hence, $f_{n}\bot f_{m}$, for $\left|m-n\right|>1$, and $0\le f-g_{n}\bot f_{m}$, for every $m<n$.

Now $f\left(x\right)>g_{n}\left(x\right)$ implies $x\notin K_{n}$, hence $h_{n}\left(x\right)<\frac{1}{n}$, therefore $f\left(x\right)-g_{n}\left(x\right)\le f\left(x\right)\le \left(f\left(x\right)-h_{n}\left(x\right)\right)^{+}+\frac{1}{n}$. If $g\ge g_{n}$ and $x\in \supp \left(f-g\right)^{+}$, then $f\left(x\right)>g_{n}\left(x\right)$, hence $x\notin K_{n}$, for every $n\in\N$; it follows that $x\notin\supp H$, and so $\left(f-g\right)^{+}\bot H$.
\end{proof}

\begin{remark}
Note that the result is true if $K$ is merely locally compact and $F\subset\Co\left(K\right)$ is a Urysohn sublattice.
\end{remark}

\section{Projection bands in vector lattices}\label{p}

In this section we prove some criterions for an ideal in a vector lattice to be a projection band.

\subsection{Locality of projection bands}

Everywhere in this section $F$ is a vector lattice. Let us show that an ideal is a projection band if and only if it is locally a projection band.

\begin{proposition}\label{pb}If $H\in \mathcal{I}_{F}$ and $g\in F$, then $H\cap I_{g}$ is a projection band in $I_{g}$ if and only if $I_{g}\cap H+I_{g}\cap H^{d}=I_{g}$, and if and only if $g\in H+H^{d}$. Hence, $H$ is a projection band if and only if $H\cap I_{g}$ is a projection band in $I_{g}$, for every $g\in G$, where $G\subset F$ is such that $G\cup H$ is majorizing.
\end{proposition}
\begin{proof}
It follows from part (i) of Corollary \ref{rel} that $\left(I_{g}\cap H\right)^{d}_{I_{g}}=I_{g}\cap H^{d}$. Hence, $H\cap I_{g}$ is a projection band in $I_{g}$ if and only if $I_{g}\cap H+I_{g}\cap H^{d}=I_{g}$. Clearly, these conditions imply that $g\in H+H^{d}$. On the other hand, if $g\in H+H^{d}$, then $H+H^{d}$ is an ideal that contains $g$, hence contains $I_{g}$, and so $I_{g}\cap H+I_{g}\cap H^{d}=I_{g}\cap \left(H+H^{d}\right)=I_{g}$.
\end{proof}

\begin{corollary}\label{pb2}
Assume that $F$ is Archimedean, and let $E$ be an order dense and majorizing sublattice of $F$. Let $H\subset F$ be an ideal such that $G=E\cap H$ is a projection band in $E$. Then, $H$ is a projection band in $F$.
\end{corollary}
\begin{proof}
First, $H^{d}\cap E\subset G^{d}_{E}$. On the other hand, if $e\in G^{d}_{E}$ and $h\in H_{+}$, then as $E$ is order dense, $h=\bigvee\left(E\cap\left[0,h\right]\right)=\bigvee\left(G\cap\left[0,h\right]\right)$. Since $G\cap\left[0,h\right]\bot e$, we conclude $e\bot h$, and thus $G^{d}_{E}=H^{d}\cap E$. As $G$ is a projection band in $E$, we have $E=G+G^{d}_{E}= H\cap E + H^{d}\cap E\subset H+ H^{d}$. Since $E$ is majorizing, and $H+H^{d}$ is an ideal that contains $E$, it follows that $F=H+ H^{d}$.
\end{proof}

Recall that a sublattice $E\subset F$ is \emph{super order dense} if for every $f\in F_{+}$ there is a countable $G\subset E$ such that $f=\bigvee G$. The following is easy to deduce from Corollary \ref{pb2}.

\begin{corollary}
If an Archimedean vector lattice has a (super) order dense and majorizing ($\sigma$-)PP sublattice, it also has ($\sigma$-)PP.
\end{corollary}

It makes sense to consider the PPP-part of a vector lattice.

\begin{corollary}\label{pppp}
$\left\{h\in F,~ I_{h}\mbox{ has PPP}\right\}$ is an ideal in $F$, which has PPP. Same for PP and $\sigma$-PP.
\end{corollary}
\begin{proof}
We will only deal with PPP. Let $H$ be the set in the claim. If $h\in H$ and $\left|f\right|\le \left|h\right|$, then $I_{f}$ is an ideal of $I_{h}$, which has PPP, from where $I_{f}$ also has PPP, and so $f\in H$. If $g,h\in H_{+}$ and $f\in I_{g+h}$, from Corollary \ref{rel} we have that
$$\left\{f\right\}^{d}_{I_{g+h}}\cap I_{g}=\left\{f\right\}^{d}\cap I_{g+h}\cap I_{g}=\left\{f\right\}^{d}\cap I_{g}=\left(I_{f}\cap I_{g}\right)^{d}_{I_{g}}=\left(f\wedge g\right)^{d}_{I_{g}}$$ is a projection band in $I_{g}$. Similarly, $\left\{f\right\}^{d}_{I_{g+h}}\cap I_{h}$ is a projection band in $I_{h}$. Since $\left\{g,h\right\}$ is a majorizing set in $I_{g+h}$, by virtue of Proposition \ref{pb} we conclude that $\left\{f\right\}^{d}_{I_{g+h}}$ is a projection band in $I_{g+h}$. Hence, $g+h\in H$, and so from Proposition \ref{id} $H$ is an ideal.\medskip

If $E$ is a principal band in $H$, it follows from part (ii) of Corollary \ref{rel} that $E\cap I_{g}$ is a principal (and hence projection) band in $I_{g}$, for every $g\in H_{+}$. Hence, according to Proposition \ref{pb}, $E$ is a projection band in $H$, and so $H$ has PPP.
\end{proof}

\begin{remark}
Note that $H:=\left\{h\in F,~ I_{h}\mbox{ has PP}\right\}$ is not necessarily a band. Indeed, $c_{0}$ is an order complete order dense ideal of $c$. If $H$ was always a band, it would have to contain $c_{0}$, hence be equal to $c$. However, $c$ does not have PPP. This can be seen directly, or using Theorem \ref{bd} below, after noticing that $2\N$ is a cozero subset of the one point compactification  $\N_{\8}$ of $\N$, whose closure is not open.
\end{remark}

\begin{theorem}\label{infd}
If $F$ is Archimedean, an ideal $H$ is a projection band if and only if $\mathcal{I}_{F}$ is infinite-meet distributive at $H$, i.e. $\bigcap\limits_{J\in \mathcal{J}}\left(H+ J\right)\subset H+ \bigcap \mathcal{J}$, for any $\mathcal{J}\subset \mathcal{I}_{F}$.
\end{theorem}
\begin{proof}
Assume that $H$ is a projection band with the projection $P$ and let $0\le f\in \bigcap\limits_{J\in \mathcal{J}}\left(H+ J\right)$. Then, for every $J\in \mathcal{J}$ there are $g_{J}\in J_{+}$ and $h_{J}\in H_{+}$ such that $g_{J}+ h_{J}=f$. Therefore, $\left\{h_{J},~ J\in \mathcal{J}\right\}\subset \left[0,f\right]\cap H$, and so $h_{J}\le Pf$, for every $J\in \mathcal{J}$; hence, $0\le f-Pf\le f-h_{J}= g_{J}\in J$, from where $f-Pf\in \bigcap \mathcal{J}$. Thus, $f\in H+ \bigcap \mathcal{J}$.\medskip

To prove the converse, it is enough to show that $H\cap I_{f}$ is a projection band in $I_{f}$, for every $f\in F$. Since $H$ is an ideal in $F$, and every ideal of $I_{f}$ is an ideal in $F$, we may assume that $F$ is a dense sublattice of $\Co\left(K\right)$, for some compact $K$, which contains $\1$. Let $U:=\supp H$. For every $x\in U$, apply the norm-dense Urysohn lemma to $\left\{x\right\}, U, \1$, and obtain $g_{x}\in F_{+}$ which vanishes at $x$ and such that $\left.g_{x}\right|_{K\backslash U}\equiv 2$. Then $M=g^{-1}_{x}\left[0,1\right]$ is a compact subset of $U$, and so applying Remark \ref{urur} to $H$ and $\1$, find $h\in H$ such that $\left.h\right|_{M}\equiv 1$. It follows that $g_{x}+h\ge \1$, and so $H+I_{g_{x}}=F$. Therefore, $\bigcap\limits_{x\in U}\left(H+ I_{g_{x}}\right)=F$, but on the other hand, $E=\bigcap\limits_{x\in U} I_{g_{x}}$ consist of functions that vanish on $U$. Thus, $H\bot E$ and $F=H+E$, and so $H$ is a projection band.
\end{proof}

Clearly, Archimedean property was not used in the proof of necessity.

\begin{question}
Is Archimedean property in Theorem \ref{infd} superfluous?
\end{question}

\subsection{Projection bands and order bounded sets}

Everywhere in this subsection $F$ is an Archimedean vector lattice. The following is a direct consequence of Lemma \ref{bern} and Krein-Kakutani theorem. We will call it the abstract ABVG Lemma.

\begin{lemma}[ABVG]
Let $f\in F_{+}$ and let $H\subset F$ be the ideal generated by an increasing sequence $\left(h_{n}\right)_{n\in\N}$. Then, there is a sequence $\left(f_{n}\right)_{n\in\N}\subset H_{+}$ with $g_{n}:=\sum\limits_{k=1}^{n}f_{k}$ such that:
\begin{itemize}
\item $f_{n}\bot f_{m}$, for $\left|m-n\right|>1$;
\item $f-g_{n}\bot f_{m}$, for every $m<n$;
\item $0\le f-g_{n}\le \left(f-h_{n}\right)^{+}+\frac{1}{n}f$, for every $n\in\N$;
\item If $g\ge g_{n}$, for every $n\in\N$, then $\left(f-g\right)^{+}\bot H$.
\end{itemize}
\end{lemma}

\begin{remark}
Note that every ideal which is generated by a countable set is generated by an increasing sequence.
\end{remark}

\begin{theorem}\label{main}For an ideal $H\subset F$ the following conditions are equivalent:
\item[(i)] $H$ is a projection band;
\item[(ii)] Every $G\subset H_{+}$, which is order bounded in $F$, is order bounded in $H$;
\item[(iii)] Every disjoint $G\subset H_{+}$, which is order bounded in $F$, is order bounded in $H$;
\item[(iv)] There is an ideal $E\subset H$ which is order dense in $H$, and such that every disjoint $G\subset E_{+}$, which is order bounded in $F$, is order bounded in $H$.\medskip

Moreover, if $H$ contains a countably generated order dense ideal $E\subset H$ such that every countable disjoint $G\subset E_{+}$, which is order bounded in $F$, is order bounded in $H$, then $H$ is a projection band.
\end{theorem}
\begin{proof}
(i)$\Rightarrow$(ii): If $P$ is the band projection corresponding to $H$, and $G\subset H_{+}$ is such that there is $f\in F_{+}$ with $G\subset\left[0,f\right]$, then $g=Pg\le Pf\in H$, for every $g\in G$. Hence, $G\subset\left[0,Pf\right]$, and so $G$ is order bounded in $H$. (ii)$\Rightarrow$(iii)$\Rightarrow$(iv) are trivial.\medskip

(iv)$\Rightarrow$(i): Let $f\in F_{+}$ and let $G\subset E_{+}$ be a maximal disjoint set. For every $g\in G$ apply ABVG Lemma to $f$ and $I_{g}$, and get $\left(f_{n}^{g}\right)_{n\in\N}\subset I_{g}\cap\left[0,f\right]\subset E\cap\left[0,f\right]$ such that any upper bound $h$ of the partial sums of this sequence satisfies $\left(f-h\right)^{+}\bot  I_{g}$. Note that $\left\{f_{2n-1}^{g},~ n\in\N,~ g\in G\right\}$ and $\left\{f_{2n}^{g},~ n\in\N,~ g\in G\right\}$ are disjoint order bounded (by $f$) subsets of $E$, and so by assumption there are upper bounds $h_{1}$ and $h_{2}$ of these sets in $H$. Then, $h=h_{1}+h_{2}\in H$ is upper bound of the partial sums of the sequence $\left(f_{n}^{g}\right)_{n\in\N}$, from where $\left(f-h\right)^{+}\bot I_{g}$, hence $\left(f-h\right)^{+}\bot g$, for every $g\in G$. Since $G$ was maximal disjoint, it follows that $f-f\wedge h=\left(f-h\right)^{+}\in E^{d}=H^{d}$. On the other hand, as $H$ is an ideal, $f\wedge h\in H$, and so $f\in H+H^{d}$. Since $f$ was arbitrary, we conclude that $H$ is a projection band.\medskip

For the last claim assume that $E$ is generated by an increasing sequence $\left(e_{n}\right)_{n\in\N}$, then repeat the proof of (iv)$\Rightarrow$(i) for the sequence $\left(f_{n}\right)_{n\in\N}$, produced by ABVG Lemma.
\end{proof}

A vector lattice is called \emph{disjointly ($\sigma$-)complete} if any (countable) disjoint order bounded set has a supremum. The second claim in the following result is from \cite{vg}.

\begin{corollary}\label{spp}
An Archimedean vector lattice $F$ has ($\sigma$-)PP if and only if for any (countable) order bounded disjoint set $G$ there is $g\in G^{dd}$ such that $G\le g$. In particular, disjoint ($\sigma$-)completeness implies ($\sigma$-)PP.
\end{corollary}
\begin{proof}
Necessity is clear. Conversely, if $H$ is a (countably generated) band in $F$, and $G\subset H$ is an order bounded (countable) disjoint set, then it is order bounded in $G^{dd}\subset H$. Hence, according to Theorem \ref{main}, $H$ is a projection band.
\end{proof}

Recall that  a Tychonoff space $X$ is called \emph{extremally (basically) disconnected} if the closure of an open (cozero) set is open. In the following theorem (i)$\Leftrightarrow$(ii) is due to Nakano (\cite{nak}), (i)$\Leftrightarrow$(iii) is due to Luxemburg and Zaanen (\cite{zl}) and Veksler (\cite{vek}), and (i)$\Leftrightarrow$(vi) is due to Veksler and Geiler (\cite{vg}).

\begin{theorem}\label{ed}For a compact Hausdorff $K$ the following conditions are equivalent:
\item[(i)]  $K$ is extremally disconnected;
\item[(ii)]  $\Co\left(K\right)$ is order complete;
\item[(iii)]  $\Co\left(K\right)$ has PP;
\item[(iv)] Every dense sublattice of $\Co\left(K\right)$ has PP;
\item[(v)] $\Co\left(K\right)$ contains a dense sublattice with PP;
\item[(vi)] $\Co\left(K\right)$ contains a dense disjointly order complete sublattice.
\end{theorem}
\begin{proof}
First, (iv)$\Rightarrow$(iii)$\Rightarrow$(v) and (ii)$\Rightarrow$(vi) are trivial, while (vi)$\Rightarrow$(v) follows from Corollary \ref{spp}. We will not prove (i)$\Rightarrow$(ii) here, referring to e.g. \cite[Theorem 43.8]{zl}.\medskip

(i)$\Rightarrow$(iv): Let $E$ be a dense sublattice. Every band in $E$ is of the form $E\left(A\right)$, where $A$ is closed. Then, $\Int A$ is clopen, and so $E\left(A\right)=E\left(\Int A\right)$ is a projection band, according to part (ii) of Proposition \ref{ur}.\medskip

(v)$\Rightarrow$(i): Let $E$ be a dense PP sublattice of $\Co\left(K\right)$. If $A\subset K$ is closed, then $E\left(A\right)$ is a band, therefore a projection band, from where $\Int A$ is clopen.
\end{proof}

\begin{remark}
In order to adapt Theorem \ref{ed} to the case when $K$ is not necessarily compact ``dense sublattice of $\Co\left(K\right)$'' should be replaced with ``dense sublattice of $\Co_{b}\left(K\right)$'' in (iii), and with ``Urysohn sublattice of $\Co\left(K\right)$'' in (iv) and (v).
\end{remark}

Attributions in the following theorem mirror those in the preceding one.

\begin{theorem}\label{bd}For a compact Hausdorff $K$ the following conditions are equivalent:
\item[(i)] $K$ is basically disconnected;
\item[(ii)] $\Co\left(K\right)$ is $\sigma$-complete;
\item[(iii)] $\Co\left(K\right)$ has PPP;
\item[(iv)] Every dense sublattice of $\Co\left(K\right)$ has $\sigma$-PP;
\item[(v)] $\Co\left(K\right)$ contains a dense disjointly $\sigma$-order complete sublattice;
\item[(vi)] $\Co\left(K\right)$ contains a dense sublattice with $\sigma$-PP.
\end{theorem}
\begin{proof}
First, (iv)$\Rightarrow$(iii), (iv)$\Rightarrow$(vi) and (ii)$\Rightarrow$(v) are trivial; while (v)$\Rightarrow$(vi) follows from Corollary \ref{spp}. Again, we refer to \cite[Theorem 43.11]{zl} for (i)$\Rightarrow$(ii).

(iii)$\Rightarrow$(i): Let $U\subset X$ be a cozero set. There is $f\in \Co\left(K\right)$ such that $U=\supp f$. Due to PPP we have that $\left\{f\right\}_{E}^{d}=\Co\left(K; K\backslash \overline{U}\right)$ is a projection band, from where $\overline{U}$ is clopen.\medskip

(i)$\Rightarrow$(iv): Let $E\subset \Co\left(K\right)$ be a dense sublattice, and $H$ is a countably generated band in $E$. That is, $H=G^{dd}_{E}$, where $G\subset E$ is countable. Hence, $U:=\supp G$ is a cozero set and such that $H=E\left(\overline{U}\right)$. Since $K$ is basically disconnected, $\overline{U}$ is clopen, and so from part (ii) of Proposition \ref{ur}, $H$ is a projection band.\medskip

(vi)$\Rightarrow$(i): Let $U$ be a cozero set. From Lemma \ref{ds2}, there is a countably generated ideal $H$ of $E$ such that $U=\supp H$. Due to $\sigma$-PP we have that $H^{d}$ is a projection band, and so $X\backslash\overline{U}=\supp H^{d}$ is clopen, according to part (ii) of Proposition \ref{ur}.
\end{proof}

\begin{remark}
It was also proven in \cite[Corollary 5.8]{erz} that a compact Hausdorff $K$ is \emph{totally disconnected} (i.e. contains no connected subsets with more than one point) if and only if $\Co\left(K\right)$ contains a dense sublattice with PPP.
\end{remark}

\begin{corollary}\label{kks}$F$ has PP ($\sigma$-PP) if and only if Krein-Kakutani spectrum of each of its elements is extremally (basically) disconnected.
\end{corollary}

We now obtain the well-known result by Luxemburg and Zaanen (see \cite{zl}).

\begin{corollary}An Archimedean vector lattice is ($\sigma$-)order complete if and only if it is PP (resp. PPP) and uniformly complete.
\end{corollary}

\subsection{About self-majorizing elements}

In this subsection we use some of the results about sublattices of $\Co\left(X\right)$ to discuss self-majorizing elements of vector lattices. The material here complements and refines \cite[Section 3.4]{weber}. Everywhere in the subsection $F$ is an Archimedean vector lattice. Recall that $e\in F_{+}$ is \emph{self-majorizing} if for every $f\in F_{+}$ there is $r>0$ such that $f\wedge ne\le re$, for every $n\in\N$. The following well-known characterization is provided for reader's convenience.

\begin{proposition}
$e$ is self-majorizing if and only if $I_{e}$ is a projection band.
\end{proposition}
\begin{proof}
If $I_{e}$ is a projection band, for every $f\in F_{+}$ there is a projection $I_{e}\ni h:=Pf=\bigvee \left(\left[0,f\right]\cap I_{e}\right)$. Therefore, for every $n\in\N$ we have $f\wedge ne\le h\le \|h\|_{e}e$. Hence, $e$ is self-majorizing.\medskip

If $e$ is self-majorizing, and $f\ge 0$, there is $r>0$ such that $f\wedge ne\le re$, for every $n\in\N$, from where $f\wedge ne=f\wedge re$, for every $n\ge r$. Thus, $I_{e}\ni f\wedge re=\bigvee \limits_{n\in\N}f\wedge ne$, and so $I_{e}$ is a projection band.
\end{proof}

\begin{proposition}If $F$ is uniformly complete, then for a countably generated ideal $I\subset F$ the following conditions are equivalent:
\item[(i)] $I$ is uniformly closed;
\item[(ii)] $I$ is a band;
\item[(iii)] $I$ is a projection band.
\end{proposition}
\begin{proof}
First, every projection band is a band, and in turn every band is uniformly closed.

Assume that $I=I\left(\left\{h_{n}\right\}_{n\in\N}\right)$ is uniformly closed, for $\left\{h_{n}\right\}_{n\in\N}\subset F_{+}$. According to Proposition \ref{pb}, it is enough to show that $H=I\left(\left\{h_{n}\right\}_{n\in\N}\right)\cap I_{f}=I\left(\left\{f\wedge h_{n}\right\}_{n\in\N}\right)$ is a projection band in $I_{f}$, for every $f\ge 0$ (equality is by part (ii) of Corollary \ref{rel}). Since $F$ is uniformly complete, $I_{f}$ is isomorphic to $\Co\left(K\right)$, for some compact Hausdorff $K$. Hence, $H$ is a closed countably generated ideal in a Banach lattice, and so is a principal ideal. Since according to Proposition \ref{pi} a closed principal ideal in $\Co\left(K\right)$ must be a projection band, the result follows.
\end{proof}

\begin{corollary}\label{sm}
If $F$ is uniformly complete, then  $e\in F_{+}$ is self-majorising if and only if $I_{e}$ is a projection band, if and only if $I_{e}$ is a band, and if and only if $I_{e}$ is uniformly closed.
\end{corollary}

\begin{remark}
Note that under PPP, the third condition is also equivalent to the first two. However, this is not true in general. Let $F$ be the space of all piecewise affine functions on $\left[-1,1\right]$. One can show that $I_{e}=F\left(\supp e\right)$, for every $e\in F$. Let $e$ be defined by $e\left(t\right)=t^{+}$. Since $\supp e=\left(0,1\right]$, which is regularly open but not clopen, $I_{e}$ is a band, but not a projection band. Taking $e$ to be the identity function shows that the last implication is also non-reversible in general.
\end{remark}

\section{Acknowledgements}

The author wants to express gratitude to Vladimir Troitsky for many valuable discussions on the topic of this paper, and in particular for contributing to the proof of Proposition \ref{dul}. Marten Wortel gets credit for bringing the author's attention to the concept of a self-majorizing element. The author also thanks Joseph van Name for contributing an idea to the proof of Theorem \ref{main}, as well as Remy van Dobben de Bruyn and KP Hart who contributed Proposition \ref{normal}, and the service \href{mathoverflow.com/}{MathOverflow} which made it possible. Finally, the anonymous reviewer provided feedback that helped to improve the exposition.

\end{document}